\documentclass[a4paper,10pt,oneside]{amsart}
\usepackage{graphicx}
\usepackage{amsfonts}
\usepackage{amsmath}
\usepackage{amssymb}
\usepackage{color} 
\usepackage[english]{babel}
\usepackage[T1]{fontenc}
\usepackage{amsthm}
\usepackage{mathrsfs}
\usepackage[noend]{algorithmic} 
\usepackage{algorithm} 
\usepackage{booktabs} 
\usepackage{cite} 
\input xy 
\xyoption{all}
\usepackage{enumitem} 

\theoremstyle{plain}
\newtheorem{thm}[algorithm]{Theorem}

\theoremstyle{definition}
\newtheorem{lem}[algorithm]{Lemma}
\newtheorem{coro}[algorithm]{Corollary}
\newtheorem{propo}[algorithm]{Proposition}
\numberwithin{algorithm}{subsection}

\newtheorem{defn}[algorithm]{Definition}

\newtheorem*{cleim}{Claim}
\newtheorem{rmk}[algorithm]{Remark}
\numberwithin{algorithm}{subsection}

\theoremstyle{remark}
\newtheorem*{ack}{Acknowledgments}
\newcommand{\Kb}{\bar{\mathbb{K}}}
\newcommand{\pp}{\mathbb{P}}
\newcommand{\esse}{\mathbb{S}}
\newcommand{\sdf}{s_{2}(\underline{F})} 
\newcommand{\odi}[1]{\mathcal{O}_{#1}}

\newcommand{\mappa}[2]{\xymatrix@1{{#1} \ar[r] & {#2}}}
\newcommand{\spazio}{\rule{1 pt}{0 cm}}
\newcommand{\gr}[1]{\mathbf{#1}}
\newcommand{\so}[1]{\underline{#1}}
\newcommand{\aprimat}[1]{\left( \begin{array}{#1}}
\newcommand{\chiudimat}{\end{array} \right)}

\newcommand{\SET}{\textbf{set }}
\newcommand{\COMPUTE}{\textbf{compute }}
\newcommand{\SUBSTITUTE}{\textbf{substitute }}
\newcommand{\CHOOSE}{\textbf{choose }}

\DeclareMathOperator{\Pf}{Pf}
\DeclareMathOperator{\depth}{depth}

\DeclareMathOperator{\im}{Im}
\DeclareMathOperator{\hd}{hd}
\DeclareMathOperator{\Ext}{Ext}
\DeclareMathOperator{\coker}{coker}
\DeclareMathOperator{\de}{d}

\DeclareMathOperator{\Vi}{V}
\DeclareMathOperator{\TG}{T}
\DeclareMathOperator{\GL}{GL}
\DeclareMathOperator{\Hh}{H}
\DeclareMathOperator{\pf}{pf}

\begin{document}

\title[Pfaffian representations of cubic surfaces]{Pfaffian representations of cubic surfaces}
\author{Fabio Tanturri}
\address{SISSA, via Bonomea 265\\
34136 Trieste, Italy}
\email{tanturri@sissa.it}
\thanks{Partially supported by the PRIN 2008 ``Geometria delle variet\`a algebriche e dei loro spazi di moduli'' and by the European Research Network ``GREFI-GRIFGA''}



\date{}

\begin{abstract}
Let $\mathbb{K}$ be a field of characteristic zero. We describe an algorithm which requires a homogeneous polynomial $F$ of degree three in $\mathbb{K}[x_{0},x_1,x_{2},x_{3}]$ and a zero $\gr{a}$ of $F$ in $\pp^{3}_{\mathbb{K}}$ and ensures a linear Pfaffian representation of $\Vi(F)$ with entries in $\mathbb{K}[x_{0},x_{1},x_{2},x_{3}]$, under mild assumptions on $F$ and $\gr{a}$. We use this result to give an explicit construction of (and to prove the existence of) a linear Pfaffian representation of $\Vi(F)$, with entries in $\mathbb{K}'[x_{0},x_{1},x_{2},x_{3}]$, being $\mathbb{K}'$ an algebraic extension of $\mathbb{K}$ of degree at most six. An explicit example of such a construction is given.
\end{abstract}

\maketitle

\section{Introduction}
Let $\mathbb{K}$ be a field of characteristic zero and let $X$ be the hypersurface in $\pp^{n}_{\mathbb{K}}$ defined by a polynomial $F \in \mathbb{K}[x_{0},x_{1},\dotsc,x_{n}]$. One may ask whether the polynomial $F^{k}$ is the determinant of a matrix $M$ with entries in $\mathbb{K}[x_{0},x_{1},\dotsc,x_{n}]$, for some integer $k$.

For $k=1$, such a matrix $M$ is said to be a \emph{determinantal representation} of $X$. If the entries are linear forms, then the determinantal representation is said to be \emph{linear}. Linear determinantal representations of curves and surfaces of small degree are a classical subject and date back to the middle of nineteenth century; see for example \cite{Beauville}, \cite{Dolgachev} for a historical account.

A relevant class of matrices with determinant $F^{2}$ are Pfaffian representations, that is, skew-symmetric matrices whose Pfaffian is $F$, up to constants. Let us recall the following definition.
\begin{defn}[Pfaffian]
Let $T=(t_{ij})$ be a skew-symmetric matrix of even size $2n$ with entries in a ring $R$. Then its determinant is the square of an element in $R$, called the \emph{Pfaffian} of $T$.\\
If we denote by $T_{ij}$ the square matrix of order $(2n-2)$ obtained by deleting from $T$ the $i$-th and $j$-th rows and columns, the Pfaffian is defined recursively as
\begin{equation}
\label{formulapfaff}
\Pf(T)=\left\{ \begin{array}{ll}
\sum_{j< 2n}(-1)^{j}t_{2n,j}\Pf(T_{2n,j})  \qquad
& \mbox{if } n \geq 2\\
t_{12} & \mbox{if } n = 1.
\end{array}
\right.
\end{equation}
\end{defn}
Pfaffian representations are a generalization of determinantal representations, since from a determinantal representation $M$ we get a Pfaffian representation
$$
\left(
\begin{array}{c|c}
0 & M \\
\hline -M & 0
\end{array}
\right).
$$
The references about Pfaffian representations are very recent, even though some general results were probably well-known to the experts before. In \cite{Beauville}, Beauville collects many results about determinantal and Pfaffian representations, giving criteria for the existence of linear Pfaffian representations of plane curves, surfaces, threefolds and fourfolds. The fact that a generic cubic threefold can be written as a linear Pfaffian had been proved by Adler \cite[Apx.V]{AdlerRamanan}, with $\mathbb{K}=\bar{\mathbb{K}}$. With the same method used by Adler, in \cite{IlievMarkushevich} it is proved that a general quartic threefold admits a linear Pfaffian representation. A non-computer-assisted proof of this fact can be found in \cite{BrambillaFaenzi}.\\
Again in the case $\mathbb{K}=\bar{\mathbb{K}}$, linear Pfaffian representations of plane curves and their elementary transformations are the subject of \cite{BuckleyKosirPla} and \cite{Buckley}; in \cite{FaenziARe} and \cite{ChiantiniFaenziOnG}, respectively almost quadratic and almost linear Pfaffian representations of surfaces are considered. In \cite{CoskunKulkarniMustopa} it is proved that every smooth quartic surface admits a linear Pfaffian representation, a result which strengthens the Beauville-Schreyer's one in \cite{Beauville}.\\

In this paper we will use the following two definitions.
\begin{defn}
Let $F \in \mathbb{K}[x_{0},x_{1},\dotsc,x_{n}]$ define the hypersurface $X$ and let $\mathbb{K}'$ be a field containing $\mathbb{K}$. A linear Pfaffian $\mathbb{K}'$-representation of $X$ is a skew-symmetric matrix whose Pfaffian is $F$, up to constants, and whose entries are linear forms in $\mathbb{K}'[x_{0},x_{1},\dotsc,x_{n}]$.
\end{defn}
\begin{defn}
[$\mathbb{K}$-point]
\label{kpoint}
If a point $\gr{a} \in \pp^{n}_{\Kb}$ admits a representative $\so{a} \in \mathbb{A}^{n+1}_{\mathbb{K}}$, then it will be called a $\mathbb{K}$-point.
\end{defn}
\noindent By convention, hypersurfaces will be considered in $\pp^{n}_{\bar{\mathbb{K}}}$, being $\bar{\mathbb{K}}$ the algebraic closure of $\mathbb{K}$. In this way, $X$ is non-empty even if its defining polynomial $F \in \mathbb{K}[x_{0},x_{1},\dotsc,x_{n}]$ has no zero in $\mathbb{A}^{n+1}_{\mathbb{K}}$, that is, if $X$ has no $\mathbb{K}$-points.\\

According to these notations, in \cite{Beauville} Beauville provided a proof of the following theorem:
\begin{thm} \label{Beavu}
Let $\esse$ be a surface of degree $d$ in $\pp^{3}_{\Kb}$, without singular $\mathbb{K}$-points. The following conditions are equivalent:
\begin{enumerate}
\item $\esse$ admits a linear Pfaffian $\mathbb{K}$-representation;
\item $\esse \cap \pp^{3}_{\mathbb{K}}$ contains a finite, reduced, arithmetically Gorenstein subscheme $Z$ of index $2d-5$, not contained in any surface of degree $d-2$.
\end{enumerate}
Moreover, the degree of $Z$ is $\tfrac{1}{6}d(d-1)(2d-1)$.
\end{thm}
\noindent Here a finite, reduced subscheme $Z$ of degree $c$ in $\pp^{n}_{\mathbb{K}}$, with ideal $I_{Z} \subset \mathbb{K}[x_{0},x_{1},\dotsc,x_{n}]$, is said to be arithmetically Gorenstein (AG for short) if $\mathbb{K}[x_{0},x_{1},\dotsc,x_{n}]/I$ is a Gorenstein ring. For such a scheme, its \emph{index} is the (unique) integer $N$ such that
\begin{equation} \label{Hilsim}
\dim \left(R/I_{Z}\right)_{p}+ \dim \left(R/I_{Z}\right)_{N-p}=c \qquad \mbox{ for all $p \in Z$}.
\end{equation}

The proof of Theorem \ref{Beavu} is based on considering the rank-two vector bundle $\coker(M)$ and its scheme $Z$ associated via the Hartshorne-Serre correspondence. As remarked by Beauville, another way to prove the existence of a Pfaffian representation is via the Buchsbaum-Eisenbud Structure Theorem in \cite{BuchsbaumEisenbud}, which we state after the following definition.
\begin{defn}[depth, Gorenstein ideal]
Let $I$ be an ideal in the ring $R$. Let $M$ be an $R$-module. Then $\depth(I,M)$ is the length of a maximal regular $M$-sequence contained in $I$.\\
The ideal $I$ is said to be Gorenstein if
\begin{equation}\label{defGore}
\depth(I,R)=\hd(R/I)=k \qquad \mbox{and} \qquad \Ext_{R}^{k}(R/I,R)\cong R/I
\end{equation}
for some $k \in \mathbb{N}$, where $\hd$ denotes the homological dimension.
\end{defn}
\begin{thm}[Buchsbaum-Eisenbud Structure Theorem] \spazio \label{BucEis}
\begin{enumerate}
\item Let $n \geq 3$ be an odd integer, and let $\mathcal{M}$ be a free module of rank $n$ over a Noetherian local ring $R$ with maximal ideal $J$. Let $f:\mappa{\mathcal{M}}{\mathcal{M}^{\ast}}$ be an alternating map of rank $n-1$ whose image is contained in $J\cdot \mathcal{M}^{\ast}$ and let $I=\Pf_{n-1}(f)$ be the ideal generated by the $(n-1) \times (n-1)$ Pfaffians of the matrix representing $f$. If $\depth(I,R)=3$, then $I$ is Gorenstein, and the minimal number of generators of $I$ is $n$.
\item Every Gorenstein ideal $I$ of $R$ with $\depth(I,R)=3$ arises as in 1..
\end{enumerate}
\end{thm}
Indeed, identifying $R$ with $\mathbb{K}[x_{0},x_{1},x_{2},x_{3}]$, an AG subscheme $Z$ as those arising in Theorem \ref{Beavu} satisfies the hypotheses of (2) in Theorem \ref{BucEis}: $Z$ has a Gorenstein homogenous ideal $I_{Z}$ by definition and by a theorem of Serre \cite{Bass}. The fact that $\depth(I_{Z},R)=3$ follows from \eqref{defGore} and $\hd(R/I_{Z})=3$, which is true since the homogeneous coordinate ring of a finite set of points is Cohen-Macaulay and from the Auslander-Buchsbaum formula \cite[ex. 18.15, ex. 19.8]{Eisenbud}.\\

Given $Z$ as in Theorem \ref{Beavu}, one can apply Theorem \ref{BucEis}: $I_{Z}$ is generated by the $(2d-2) \times (2d-2)$ principal Pfaffians extracted from a skew-symmetric $(2d-1)\times (2d-1)$ matrix $T$ with linear forms as entries. Then the surface admits a Pfaffian $\mathbb{K}$-representation
\begin{equation}
\label{facile}
\left(
\begin{array}{c|c}
T & -C^{t} \\
\hline C & 0
\end{array}
\right),
\end{equation}
where $C$ is a suitable $1 \times (2d-1)$ matrix with linear forms as entries, which can be found by formula \eqref{formulapfaff} (see also subsection \ref{analogo}).\\

In this paper we focus on case $d=3$. If $\mathbb{K}=\bar{\mathbb{K}}$, then by \cite{DavisGeramitaOrecchia} a set of five points in $\pp^{3}_{\mathbb{K}}$ is an AG scheme if and only if they are in general position, i.e. no four of them are on a plane. This fact, together with Theorem \ref{BucEis}, implies
\begin{coro} \label{coroler}
If $\mathbb{K}=\bar{\mathbb{K}}$, every smooth cubic surface in $\pp^{3}_{\mathbb{K}}$ admits a linear Pfaffian representation \cite{Beauville}.
\end{coro}

This result has been generalized in \cite{FaniaMezzetti} as follows.
\begin{propo}
\label{ognisuperficie}
If $\mathbb{K}=\bar{\mathbb{K}}$, every cubic surface in $\pp^{3}_{\mathbb{K}}$ admits a linear Pfaffian representation.
\end{propo}

We study how to construct \emph{explicitly} a linear Pfaffian $\mathbb{K}$-representation, where $\mathbb{K}$ is not necessarily algebraically closed, starting from the least amount of initial data possible. We will show that, in general, it is sufficient to know a $\mathbb{K}$-point on $\esse$.\\
Our contribution is the following: we prove %
\begin{thm}
\label{mainthm}
Let $\esse$ be a cubic surface, neither reducible nor a cone, whose equation is $F \in \mathbb{K}[x_{0},x_{1},x_{2},x_{3}]_{3}$. Given a $\mathbb{K}$-point $\gr{a^{1}}$, which is not a T-point --- in the sense of Definition \ref{Tpoint} --- it is possible to construct explicitly a linear Pfaffian $\mathbb{K}$-representation of $\esse$.
\end{thm}
\noindent The same method can be used to prove a weaker result, if $\gr{a^{1}}$ is not given:
\begin{propo}
\label{propouno}
Let $\esse$ be a cubic surface, neither reducible nor a cone, whose equation is $F \in \mathbb{K}[x_{0},x_{1},x_{2},x_{3}]_{3}$. Then it is possible to construct explicitly a Pfaffian $\mathbb{K'}$-representation of $\esse$, where $\mathbb{K'}$ is an algebraic extension of degree ${[\mathbb{K}':\mathbb{K}]\leq3}$.\\
Moreover, if $\mathbb{K}\subseteq \mathbb{R}$, then also $\mathbb{K}'$ can be chosen so.
\end{propo}
On one hand, these results strengthen one implication of Theorem \ref{Beavu} and give a bound for the degree of algebraic extension required to produce a linear Pfaffian representation. On the other hand, they are constructive: it is possible to implement an algorithm which produces a linear Pfaffian representation, provided the requested inputs.\\
After discussing the cases of reducible surfaces and cones, we are able to prove Theorem \ref{thmulti}, which strengthens Proposition \ref{ognisuperficie}.
\begin{thm}
\label{thmulti}
Every cubic surface in $\pp^{3}_{\bar{\mathbb{K}}}$, with equation $F \in \mathbb{K}[x_{0},x_{1},x_{2},x_{3}]_{3}$, admits a Pfaffian $\mathbb{K}'$-representation, $\mathbb{K}'$ being an algebraic extension of $\mathbb{K}$ of degree ${[\mathbb{K}':\mathbb{K}]}\leq6$.\\
Moreover, it is possible to explicitly realize such a representation.

\end{thm}
This paper is structured as follows: in section \ref{one}, we retrace the proof of Theorem \ref{BucEis} and we use it to construct a skew-symmetric matrix $\mathbb{T}$ as in \eqref{ti}, whose Pfaffians generate the ideal of the four fundamental points and the unit point in $\pp^{3}$. This enables us to produce Algorithm \ref{algo}, whose inputs are five points in general position on a surface $\esse$ and whose output is a linear Pfaffian representation of $\esse$.\\
In section \ref{two}, we make use of the tangent plane process, a classical argument (see, for example, \cite{Segre}); starting from a $\mathbb{K}$-point $\gr{a^{1}}$ on an irreducible surface which is not a cone, we show that it is always possible to find four other points on the surface such that all the five points are in general position, provided that $\gr{a^{1}}$ satisfies a mild condition.\\
In section \ref{three} we summarize the previous results in Theorem \ref{mainthm} and Proposition \ref{propouno}. Then we discuss the case of reducible surfaces and the case of cones, so to prove Theorem \ref{thmulti}. An example of the construction of a Pfaffian representation is finally given.

\section{From five points to a Pfaffian representation}
\label{one}
In this section, we make explicit the construction of the proof of Theorem \ref{BucEis}, in the particular case of the
ideal $I$ of the four fundamental points and the unit point
\begin{equation}
\label{ipunti}
[1:0:0:0],[0:1:0:0],[0:0:1:0],[0:0:0:1],[1:1:1:1]
\end{equation}
in $\pp^{3}_{\mathbb{Q}}$. This produces the skew-symmetric matrix $\mathbb{T}$ in \eqref{ti}, whose Pfaffians generate $I$; we will make use of $\mathbb{T}$ to implement Algorithm \ref{algo}, which produces a linear Pfaffian $\mathbb{K}$-representation of a cubic surface $\esse$ starting from five $\mathbb{K}$-points in general position on $\esse$.\\
From now on, we will consider only \emph{linear} Pfaffian representations.
\subsection{An explicit construction}
For the sake of completeness, we recall briefly the constructions made in \cite{BuchsbaumEisenbud} in the proof of Theorem \ref{BucEis}.

Let $R$ be the ring of polynomials $\mathbb{K}[x_{0},x_{1},x_{2},x_{3}]$ and let $I$ be a Gorenstein ideal with $\depth(I,R)=3$.
From a minimal free resolution of $I$\begin{equation}
\label{risdii}
\xymatrix@1{
\underline{F}: \qquad 0 \ar[r] & F_{3} \ar[r]^{\de_{3}} & F_{2} \ar[r]^{\de_{2}} & F_{1} \ar[r]^{\de_{1}} & F_{0} \ar[r] & R/I \ar[r] & 0
}\mbox{,}
\end{equation}
where $F_{0} \cong R \cong F_{3}$, it is possible to make a change of basis in $F_{1}$ such that the map $\mappa{F_{2}}{F_{1}}$ is alternating. This can be found by equipping this resolution with a graded commutative algebra, the symmetric square of $\underline{F}$
$$\sdf = {(\underline{F}\otimes \underline{F})}/{M}\mbox{,}$$
where $M$ is the graded submodule of $\underline{F}\otimes \underline{F}$ generated by the elements of the set
$$\left\{ f \otimes g - (-1)^{(\deg f)(\deg g)}g \otimes f \:|\: f,g \mbox{ homogeneous elements of \underline{F}}\right\}\mbox{.}$$
By convention, an element $f$ has degree $i$ if and only if it belongs to $F_{i}$; the degree of $(f \otimes g)$ is simply $\deg(f)+\deg(g)$. The differential is inherited from $\underline{F}$ as follows:
$$
\de(f\otimes g)= \de f \otimes g + (-1)^{\deg f}f \otimes \de g.
$$
The symmetric square $\sdf$ is a complex of projective $R$-modules, canonically isomorphic to $\underline{F}$ in degree 0 and 1. Therefore, there exists a map of complexes $\Phi:\mappa{\sdf}{\underline{F}}$ which lifts up these two isomorphisms and it can be chosen so that the restrictions of $\Phi$ to $F_{0} \otimes F_{k}$ are the isomorphisms $F_{0} \otimes F_{k} \cong F_k$.\\
The multiplication in $\sdf$ is given by $f \cdot g= \Phi(\overline{f \otimes g})$, where $\overline{f \otimes g}$ is the class of $f \otimes g$ modulo $M$. Since $F_{3} \cong R$, this multiplication induces a map $\mappa{F_{k}\otimes F_{3-k}}{R}$, which turns to be a perfect pairing.\\
This can be viewed as an isomorphism between $F_{1}$ and ${F_{2}}^{\ast}$, which makes the composition $\xymatrix@1{F_{2} \ar[r] & F_{1} \ar[r] & F_{2}^{\ast}}$ an alternating map.\\

Let us consider the special case where $I$ is the ideal of the points \eqref{ipunti}. We have the free resolution \eqref{risdii}, with $F_{1} \cong R^{5} \cong F_{2}$. We have to develop $\Phi_{3}:\mappa{\sdf_{3}}{F_{3}}$ in the diagram
\begin{equation}
\label{sollevamento}
\xymatrix{
\dotso \ar[r] & \sdf_{3} \ar[d]^{\Phi_{3}} \ar[r]^{d'_{3}} & \sdf_{2} \ar[d]^{\Phi_{2}} \ar[r]^{d'_{2}} & \sdf_{1} \ar[d]^{\Phi_{1}} \ar[r]^{d'_{1}} & \sdf_{0} \ar[d]^{\Phi_{0}} \ar[r]^{\pi} & R/I \ar[r] & 0 \\
0 \ar[r] & F_{3} \ar[r]^{d_{3}} & F_{2} \ar[r]^{d_{2}} & F_{1} \ar[r]^{d_{1}} & F_{0} \ar[r]^{\pi} & R/I \ar[r] & 0
}
\end{equation}
We choose the ordered basis of $\sdf_{2} \cong (F_{0}\otimes F_{2}) \oplus (\wedge^{2} F_{1})$ to be formed by the classes modulo $M$ of $  1 \otimes f^{2}_{1}, 1 \otimes f^{2}_{2}, \dotsc , 1 \otimes f^{2}_{5}, f^{1}_{1} \otimes f^{1}_{2}, f^{1}_{1} \otimes f^{1}_{3}, \dotsc, f^{1}_{4} \otimes f^{1}_{5}$, where the $f^{1}_{i}$s are a basis of $F_{1}$ and the $f^{2}_{j}$s are a basis of $F_{2}$. A similar convention is fixed for $\sdf_{3}\cong (F_{0}\otimes F_{3}) \oplus (F_{1} \otimes F_{2})$.\\

After a computation with \cite{CoCoA}, we consider the maps of diagram \eqref{sollevamento} to be
$$d_{3}=\aprimat{c}
x_0x_1-x_1x_3 \\
x_1x_2-x_2x_3 \\
-x_0x_2+x_1x_2 \\
-x_1x_3+x_2x_3 \\
x_0x_3-x_1x_3
\chiudimat\mbox{,} \qquad \qquad
{d_{1}}^{t}={d'_{1}}^{t}=\aprimat{c}
x_1x_3-x_2x_3 \\ x_0x_3-x_2x_3 \\ x_1x_2-x_2x_3 \\ x_0x_2 - x_2x_3 \\ x_0x_1 - x_2x_3
\chiudimat\mbox{,}
$$
$$d_{2}=\aprimat{ccccc}
-x_2&x_0&0&0&x_2\\
x_2&-x_1&x_1&0&0\\
x_3&-x_3&x_3&x_0-x_3&0\\
-x_3&x_3&0&-x_1+x_3&x_1\\
0&0&-x_3&0&-x_2
\chiudimat\mbox{.}
$$
The isomorphisms $\Phi_{0}$ and $\Phi_{1}$ are represented by identity matrices. With straightforward computations we get the matrices ${d'_{2}}$ and $d'_{3}$. By trials, we can lift up $\Phi_{1}$ by finding matrices $\Phi_{2}$ and $\Phi_{3}$ such that the diagrams
$$
\xymatrix{
&\sdf_{2} \ar[d]^{\Phi_{1} \circ d'_{2}} \ar
[dl]^{\Phi_{2}}&\\
F_{2} \ar[r]^{d_{2}} & \im(d_{2}) \ar[r] & 0
}
\qquad \qquad
\xymatrix{
&\sdf_{3} \ar[d]^{\Phi_{2} \circ d'_{3}} \ar
[dl]^{\Phi_{3}}&\\
F_{3} \ar[r]^{d_{3}} & \im(d_{3}) \ar[r] & 0
}
$$
commute. A possible choice for $\Phi_{2}$ is
\begin{equation*}
\aprimat{c|c}
I_{5} & \begin{array}{cccccccccc}
-x_3&x_1&0&0&x_3&x_3-x_0&x_3&0&-x_1&0\\
-x_3&0&-x_2&-x_1&0&-x_2&0&0&0&x_2\\
0&-x_2&-x_2&-x_1&-x_2&-x_2&-x_0&0&x_2&x_2\\
0&0&0&0&x_3&x_3&x_3&-x_2&-x_1&0\\
0&x_3&x_3&x_3&x_3&x_3&x_3&0&-x_1&-x_0
\end{array}
\chiudimat \mbox{.}
\end{equation*}
This choice is indeed the unique with linear forms as entries in the right block, since the syzygies are of degree two. The map $\Phi_{3}$ turns to be
$$
\renewcommand\arraycolsep{1.1mm}
\aprimat{c|c}
I_{1} & \begin{array}{ccccccccccccccccccccccccc}
0&0&0&-1&0&0&0&0&-1&1&0&1&0&0&0&0&1&-1&0&0&1&0&0&-1&0
\end{array}
\chiudimat \mbox{.}
$$

The isomorphism resulting from $\Phi_{3}$ is
$$
\xymatrix@1{{F_{1}} \ar[rrrrrr]^{\aprimat{ccccc}
0&0&0&0&1\\
0&0&1&1&0\\
0&0&0&-1&0\\
-1&-1&0&0&-1\\
0&1&0&0&0
\chiudimat} &&&&&&{{F_{2}}^{\ast}}}
$$
and, with respect to this change of basis, the map $d_{2}$ turns to be alternating, represented by the skew-symmetric matrix
\begin{equation}\label{ti}
\mathbb{T}=\aprimat{ccccc}
0&0&-x_3&0&-x_2\\
0&0&x_3&x_0-x_1&x_1\\
x_3&-x_3&0&x_1-x_3&-x_1\\
0&-x_0+x_1&-x_1+x_3&0&0\\
x_2&-x_1&x_1&0&0
\chiudimat\mbox{.}
\end{equation}
It is easy to verify that the $4 \times 4$ principal Pfaffians of $\mathbb{T}$ --- listed in \eqref{pfaffianiT} --- are exactly the five generators of $I$, that is, the entries of $d_{1}$.

\subsection{From five points to a Pfaffian representation: an algorithm}
\label{analogo}
The procedure just shown can be applied as long as we have the ideal of a set $X$ of five points in general position on a cubic surface $\esse$. Due to the classical fact that two sets of five points in general position in $\pp^{3}$ are projectively equivalent, instead of repeating the previous construction it is also possible to realize a Pfaffian representation in the following way.\\
By solving a linear system, we can find the matrix $A$ of the projectivity which maps $X$ to the five points \eqref{ipunti}. Replacing $x_{0},x_{1},x_{2},x_{3}$ in \eqref{ti} with the columns of the matrix $\left(\begin{array}{cccc}x_{0} & x_{1} & x_{2} & x_{3}\end{array}\right) \cdot A^{t}$, we get a matrix $T$ whose Pfaffians $P_{i}$ generate the ideal of $X$.\\
Finding a Pfaffian representation is then straightforward: if $\esse = \Vi(F)$, then $F$ belongs to the ideal of $X$. Therefore, one can find five linear forms $L_{i}$ such that $F=\sum_{i=1}^{5}(-1)^{i+1}L_{i}P_{i}$. Setting $C=\left(\begin{array}{ccccc} L_{1} & L_{2} & L_{3} & L_{4} & L_{5} \end{array} \right)$ and by \eqref{formulapfaff}, we get a Pfaffian representation of the form \eqref{facile}.\\
We summarize the whole procedure in Algorithm \ref{algo}
, presented in pseudocode, where $\mathbb{T}=\mathbb{T}(x_{0},x_{1},x_{2},x_{3})$ in \eqref{ti} is seen as a matrix depending on four variables, the Pfaffians of which are
\begin{gather}
\begin{gathered}
\label{pfaffianiT}
\Pf_{1}(\mathbb{T})(x_{0},x_{1},x_{2},x_{3})=x_1(x_0-x_3) \\
\Pf_{2}(\mathbb{T})(x_{0},x_{1},x_{2},x_{3})=x_{2}(x_3-x_1) \\
\Pf_{3}(\mathbb{T})(x_{0},x_{1},x_{2},x_{3})=x_{2}(x_1-x_0) \\
\Pf_{4}(\mathbb{T})(x_{0},x_{1},x_{2},x_{3})=x_3(x_1-x_{2}) \\
\Pf_{5}(\mathbb{T})(x_{0},x_{1},x_{2},x_{3})=x_3(x_0-x_1).
\end{gathered}
\end{gather}

\algsetup{indent=5em}
\begin{algorithm}[h!]
\caption{from five points in general position to a Pfaffian representation}
\label{algo}
\begin{algorithmic}[1]
\renewcommand{\zeta}{z}
\REQUIRE $F \in {\mathbb{K}[x_{0},x_{1},x_{2},x_{3}]}_{3}$ and $\gr{a^{1}},\gr{a^{2}},\gr{a^{3}},\gr{a^{4}},\gr{a^{5}}$ $\mathbb{K}$-points in general position on $\esse=\Vi(F)$
\ENSURE $M$, a Pfaffian $\mathbb{K}$-representation of $\esse$ depending on some arbitrary parameters $\alpha_{i,j}$
\STATE \CHOOSE a representative $\so{a^{i}}=(a^{i}_{0},a^{i}_{1},a^{i}_{2},a^{i}_{3}) \in \mathbb{A}^{4}_{\mathbb{K}}$ of $\gr{a^{i}}$ for every $1 \leq i \leq 5$
\STATE \COMPUTE the solution $\so{\lambda}=(\lambda_{1},\lambda_{2},\lambda_{3},\lambda_{4})$ of the linear system $$
\left(
\begin{array}{cccc}
a^{1}_{0} & a^{2}_{0} & a^{3}_{0} & a^{4}_{0} \\
a^{1}_{1} & a^{2}_{1} & a^{3}_{1} & a^{4}_{1} \\
a^{1}_{2} & a^{2}_{2} & a^{3}_{2} & a^{4}_{2} \\
a^{1}_{3} & a^{2}_{3} & a^{3}_{3} & a^{4}_{3}
\end{array}
\right)
\left(
\begin{array}{c}
\lambda_{1} \\
\lambda_{2} \\
\lambda_{3} \\
\lambda_{4}
\end{array}
\right)
=
\left(
\begin{array}{c}
a^{5}_{0} \\
a^{5}_{1} \\
a^{5}_{2} \\
a^{5}_{3}
\end{array}
\right)$$
\STATE \COMPUTE the change of basis matrix $A$ from $(\lambda_{i} a^{i})_{1\leq i \leq 4}$ to the standard basis of $\mathbb{A}^{4}_{\mathbb{K}}$, so that
$$
\lambda_{i} A
\left(
\begin{array}{c}
a^{i}_{0} \\
a^{i}_{1} \\
a^{i}_{2} \\
a^{i}_{3}
\end{array}
\right)=
\left(
\begin{array}{c}
\delta_{i}^{1}\\
\delta_{i}^{2} \\
\delta_{i}^{3} \\
\delta_{i}^{4}
\end{array}
\right) \qquad \mbox{for every } 1 \leq i \leq 4
$$

\FOR{$i=1$ \TO $4$}
	\STATE \SET $\zeta_{i-1}$ as the $i$-th row of the column vector $A \cdot \left(\begin{array}{c}x_{0} \\ x_{1} \\ x_{2} \\ x_{3}\end{array}\right)$
	\ENDFOR
\STATE \SET $\mathbb{T}(x_{0},x_{1},x_{2},x_{3})$ as in \eqref{ti}
\STATE \SET $T=\mathbb{T}(\zeta_{0},\zeta_{1},\zeta_{2},\zeta_{3})$
\FOR{$i=1$ \TO $5$} \STATE \label{cambiodd} \SET $P_{i}=\Pf_{i}(\mathbb{T})(\zeta_{0},\zeta_{1},\zeta_{2},\zeta_{3})$ as in \eqref{pfaffianiT}
\STATE \SET $L_{i}=\sum_{j=0}^{3}\alpha_{i,j}x_{j}$
\ENDFOR
\STATE \SET $G=F-\sum_{i=1}^{5}(-1)^{i+1}L_{i}P_{i}$
\STATE \label{trova} \COMPUTE solutions of the linear system given by equaling the coefficients of $G$ to zero, $\alpha_{i,j}$ as unknowns
\STATE \SUBSTITUTE the solutions in $L_{i}$
\STATE \SET $M$ as the matrix $$\left( \begin{array}{c|c}
T &
\begin{array}{c}
L_{1} \\ L_{2} \\ L_{3} \\ L_{4} \\ L_{5}
\end{array}
\\
\hline 
\begin{array}{ccccc}
-L_{1} & -L_{2} & -L_{3} & -L_{4} & -L_{5}
\end{array}
& 0
\end{array} \right)$$
\end{algorithmic}
\end{algorithm}
\begin{rmk}
\label{remarco}
Algorithm \ref{algo} involves only linear equations. If the five given points are $\mathbb{K}$-points, as well as $F \in \mathbb{K}[x_{0},x_{1},x_{2},x_{3}]_{3}$, then the output Pfaffian representation of $\esse=\Vi(F)$ is a $\mathbb{K}$-representation too, for a suitable choice of the representatives of the points.
\end{rmk}

\begin{rmk}
\label{fivedime}
The matrix associated to the (non-homogeneous) linear system in line \ref{trova} of the algorithm is $20 \times 20$; it depends only on the projectivity applied in line \ref{cambiodd}, not on the choice of $F$. Regardless to this projectivity, its rank is $15$.\\
Since for any choice of $F \supset \{\gr{a^{1}},\gr{a^{2}},\gr{a^{3}},\gr{a^{4}},\gr{a^{5}}\}$ a solution of this linear system does exist, the ``Pfaffian representation depending on some parameters'' ensured by Algorithm \ref{algo} turns to be a five-dimensional linear space of Pfaffian representations.
\end{rmk}

\subsubsection{Classes of equivalent representations}

We recall that two Pfaffian $\mathbb{K}$-representations $M$ and $M'$ are said to be equivalent if and only if there exists $X \in \GL_{\mathbb{K}}(6)$ such that $M'=XMX^{t}$. Let $\coker(M)$, $\coker(M')$ be the cokernel sheaves of $M,M'$, seen as maps $\mappa{\odi{\pp^{3}}^{6}(-1)}{\odi{\pp^{3}}^{6}}$. Then from \cite[(2.3)]{Beauville} it follows that $M,M'$ are equivalent if and only if $\coker(M) \cong \coker(M')$.\\
In this way the study of equivalence classes of Pfaffian representations of a cubic surface $\esse$ is strongly linked to the study of certain sheaves on $\esse$.\\

\begin{rmk}
\label{tutteequiv}
Let $Z$ be a fixed set of five points in general position on a surface $\esse$ without singular $\mathbb{K}$-points and consider the Pfaffian representations given by Algorithm \ref{algo}, which are a five-dimensional linear space by Remark \ref{fivedime}. It turns out that all these representations are equivalent. Indeed, by \cite[(7.1)]{Beauville}, up to automorphism there exists only one pair $(E,s)$, with $E$ rank-two vector bundle on $\esse$ and $s \in \Hh^{0}(\esse,E)$, such that $Z$ is the zero locus of $s$. In addition, these classes of pairs $[(E,s)]$ are in bijection with the equivalence classes of the pairs $[(M,\bar{s})]$, where $E=\coker(M)$ and $\bar{s} \in \Hh^{0}(\pp^{3},\odi{\pp^{3}}^{6})$ corresponds to $s$ via the isomorphism $\Hh^{0}(\pp^{3},\odi{\pp^{3}}^{6}) \cong \Hh^{0}(\esse,E)$. It follows that all the representations produced in the algorithm belong to a unique equivalence class.
\end{rmk}

\noindent It is worth noting that, as $Z$ varies among the possible sets of five $\mathbb{K}$-points in general position on a surface $\esse$ without singular $\mathbb{K}$-points, Algorithm \ref{algo} is surjective onto the possible Pfaffian $\mathbb{K}$-representations of $\esse$, and therefore onto their equivalence classes. Indeed, as shown in \cite[(7.2)]{Beauville}, a general global section of $E=\coker(M)$ has five points in general position as its zero locus $Z$ and therefore $M$ can be produced via the algorithm with input $Z$.\\

In \cite{Buckley}, elementary transformations were used to construct non-equivalent Pfaffian representations of curves starting from a given one. This technique can be used in the case of surfaces as well.
\begin{rmk}
\label{neaspetto5}
The bijection between Pfaffian representations $M$ and sheaves $E=\coker(M)$ tells us more, when dealing with the algebraic closure $\Kb$.\\
Let $\mathscr{S}^{3}/\GL(6)$ be the set of equivalence classes of the $6 \times 6$ skew-symmetric matrices of linear forms in $\pp_{\Kb}^{3}$; let $\pf:\mappa{\mathscr{S}^{3}/\GL(6)}{|\odi{\pp^{3}_{\Kb}}(3)|}$ be the map which associates to a class $[M]$ the cubic surface in $\pp^{3}_{\Kb}$ with equation $\Pf(M)$. As noticed in \cite{FaniaMezzetti}, for the general $\esse$ the fiber $\pf^{-1}(\esse)$ can be identified with an open subset of the moduli space of simple rank-two vector bundles $E$ on $\esse$ with $c_{1}(E)=\odi{\esse}(2)$,  $c_{2}(E)=5$.\\
Since the (projective) dimension of $\mathscr{S}^{3}/\GL(6)$ is $59-35=24$ and the dimension of $|\odi{\pp^{3}_{\Kb}}(3)|$ is $19$, then for the general $\esse$ we have a five-dimensional space of essentially different Pfaffian $\Kb$-representations of $\esse$.
\end{rmk}
%

The space $\mathscr{S}^{3}/\GL(6)$ has been recently considered in \cite{Han}, in relation to the space of pairs $(\esse, \Pi)$, being $\Pi$ a complete pentahedron inscribed in $\esse$.

\section{Constructing five points on a surface}
\label{two}
Given an equation $F \in \mathbb{K}[x_{0},x_{1},x_{2},x_{3}]_{3}$, in general it is not easy to find a zero of $F$ in $\mathbb{A}^{4}_{\mathbb{K}}$. For example, if $\mathbb{K}=\mathbb{Q}$, the problem of the existence of rational points on cubic surfaces, reliable to diophantine equations, has been strongly faced in the last century (see, for example, \cite{Mordell}, \cite{Segre} and the more recent \cite{ManinCub}).\\
Our next aim is to weaken the required inputs of Algorithm \ref{algo}.
\subsection{From one point to five points}
It is well known that from a \emph{general} choice of a $\mathbb{K}$-point on a \emph{general} cubic surface with equation in $\mathbb{K}[x_{0},x_{1},x_{2},x_{3}]_{3}$ it is possible to find infinitely many others $\mathbb{K}$-points on the surface; this can be performed by using the tangent plane process, a classical argument (for example, see \cite{Segre}). It starts by taking the tangent plane to the cubic surface $\esse$ at a smooth point $P$. $\TG_{P}\esse$ cuts $\esse$ in a curve of degree three, for which $P$ is a singular point. A line through $P$, lying on the tangent plane, intersects $\esse$ twice in $P$, while the third intersection is \emph{generically} different and gives us another $\mathbb{K}$-point on $\esse$.\\
We want to get rid of this ``generality''. Theorem \ref{teorema} will show how, under reasonable hypotheses, the tangent plane process applied to a starting $\mathbb{K}$-point can be repeated to produce four other $\mathbb{K}$-points on $\esse$, such that the five points are all together in general position. This will prove, under these hypotheses, that we only need a $\mathbb{K}$-point on $\esse$ to construct an explicit Pfaffian $\mathbb{K}$-representation.
\begin{defn}
\label{Tpoint}
A point $P \in \esse$ will be called a T-point for $\esse$ if $P$ is smooth for $\esse$ and $\TG_{P}\esse \cap \esse$ is set-theoretically union of lines.
\end{defn}
Let us observe that the so-called Eckardt points, i.e. smooth points $P$ with $\TG_{P}\esse \cap \esse$ made up of three lines through $P$, are T-points. Moreover, a smooth points $P$ is a T-point if and only if $\TG_{P}\esse$ is a tritangent plane.\\
In general, for a T-point $P$ one expects $\TG_{P}\esse \cap \esse$ to be union of three distinct lines, but it is possible to have one line with multiplicity three or two lines, one of them with multiplicity two.\\
The role of T-points will be clear in a while. Let us remark that, for a smooth point $P$ which is not a T-point, $\TG_{P}\esse \cap \esse$ is either an irreducible cubic curve with $P$ as a singular point, or union of a line through $P$ and a smooth conic passing through $P$.
\begin{rmk}
\label{surette}
Let $P$ be a T-point for $\esse$. If $\TG_{P}\esse \cap \esse$ is a line $r$ with multiplicity three, or union of a line $r$ with multiplicity two and another line, then $r$ is union of singular points for $\esse$ and T-points for $\esse$ sharing the same tangent plane.
\end{rmk}
\begin{thm}
\label{teorema}
Let $\esse$ be an irreducible cubic surface which is not a cone, whose equation is $F \in \mathbb{K}[x_{0},x_{1},x_{2},x_{3}]_{3}$. Given a $\mathbb{K}$-point $\gr{a^1}$ on $\esse$ which is not a T-point --- in the sense of Definition \ref{Tpoint} --- it is possible to explicitly construct four other $\mathbb{K}$-points on $\esse$ such that the five points together are in general position.
\end{thm}

\noindent The constructive proof, which requires some steps and preliminary lemmas, will be the subject of next subsection. In subsection \ref{commenti} we will see how this construction can be adapted if some of the hypotheses are missing.
\subsection{
}
Let us consider $F=F(x_{0},x_{1},x_{2},x_{3}) \in \mathbb{K}[x_{0},x_{1},x_{2},x_{3}]_{3}$. Then we set, for every $\so{a}=(a_{0},a_{1},a_{2},a_{3}) \in \mathbb{A}^{4}_{\Kb}$:
\begin{itemize}
\item $P_{1,\so{a}}(x_{0},x_{1},x_{2},x_{3})=\sum_{i=0}^{3}a_{i} \frac{\partial F}{\partial x_{i}}$;
\item $P_{2,\so{a}}(x_{0},x_{1},x_{2},x_{3})=\sum_{i=0}^{3}x_{i}\frac{\partial F}{\partial x_{i}}(\so{a})$.
\end{itemize}
They are the equations of the first polar and the second polar of $\gr{a}={[a_{0}:a_{1}:a_{2}:a_{3}]}$ with respect to the surface $\esse=\Vi(F)$. If $\gr{a}$ is smooth, $P_{2,\so{a}}$ defines $\TG_{\gr{a}}\esse$.\\
If $\so{x}=(x_{0},x_{1},x_{2},x_{3})$, for every $\so{a} \in \mathbb{A}^{4}_{\Kb}$ we have:
\begin{equation}
\label{sviluppo}
F(\so{a}+t\so{x})=F(\so{a})+tP_{2,\so{a}}(\so{x})+t^{2}P_{1,\so{a}}(\so{x})+t^{3}F(\so{x}).
\end{equation}
We will consider the first and the second polar $\Vi(P_{1,\gr{a}})$ and $\Vi(P_{2,\gr{a}})$, for $\gr{a} \in \pp^{3}_{\Kb}$, as hypersurfaces in $\pp^{3}_{\Kb}$.
\begin{lem}
\label{lemmacono}
Let $\gr{a}$ be a singular point on a cubic surface $\esse$, whose equation is $F \in \mathbb{K}[x_{0},x_{1},x_{2},x_{3}]_{3}$. Let us assume that $\esse$ is neither reducible, nor a cone. Then there are at most six lines through $\gr{a}$ lying on $\esse$.
\begin{proof}
By \eqref{sviluppo}, if a point $\gr{x} \in \esse \cap \Vi(P_{1,\gr{a}})$, also the whole line through $\gr{a}$ and $\gr{x}$ does. $P_{1,\gr{a}}$ is not the zero polynomial since $\esse$ is not a cone, moreover $F$ is irreducible: this means that the intersection $\esse \cap \Vi(P_{1,\gr{a}})$ is transversal. It is therefore a curve of degree six, union of lines through $\gr{a}$.
\end{proof}
\end{lem}
\begin{lem}
\label{lemmaTpoints}
Let $\esse$ be an irreducible, cubic surface which is not a cone and let us assume $\gr{a} \in \esse$ is not a T-point.
\begin{enumerate}
\item If $\gr{a}$ is smooth, then on $\TG_{\gr{a}}\esse$ there are only finitely many T-points for $\esse$. Moreover $\Vi(P_{1,\gr{a}}) \cap \TG_{\gr{a}}\esse$ is union of at most two lines through $\gr{a}$ and any line through $\gr{a}$ lying on $\esse$ lies also on $\Vi(P_{1,\gr{a}}) \cap \TG_{\gr{a}}\esse$.
\item If $\gr{a}$ is singular, then point 1. still holds if we replace $\TG_{\gr{a}}\esse$ with a plane $\pi$ through $\gr{a}$, for all but finitely many choices of $\pi$.
\end{enumerate}
\begin{proof}
We distinguish two classes of T-points: let us call $\mathcal{A}$ the set of T-points $P$ for which $\TG_{P}\esse \cap \esse$ is union of three distinct lines, $\mathcal{A}'$ the set of T-points not in $\mathcal{A}$.\\
Either $\esse$ contains finitely many lines or infinitely many ones. In the first case, note that $\mathcal{A}$ is a finite set, since mutual intersections of lines on $\esse$ are finite in number; $\mathcal{A'}$ is contained in a union of lines on $\esse$, by Remark \ref{surette}.\\
\noindent If $\esse$ contains infinitely many lines, then it is well-known (for example, \cite{Conforto}) that $\esse$ is either reducible, an irreducible cone or a ruled cubic with a double line. By hypotheses the first two cases cannot occur.\\
Moreover, a cubic surface with a double line which is not a cone is projectively equivalent to either ${\Vi(x_{0}^{2}x_{3}^{}+x_{0}^{}x_{1}^{}x_{2}^{}+x_{1}^{3})}$ or ${\Vi(x_{0}^{2}x_{2}^{}+x_{1}^{2}x_{3}^{})}$ (see, for example, \cite{Abhyankar}). The study of these two cases leads to Table \ref{tabella0} and Table \ref{tabella}.\\
\begin{table}[hbt]
\centering
\begin{tabular}{cc}
\toprule
coordinates of $\gr{a}$ & $\TG_{\gr{a}}\esse \cap \esse$ (if smooth)\\
\toprule
$[1:s:t:-s^{3}-st]$ & $\left\{\begin{array}{l}
x_{0}(-2s^{3}-st)+x_{1}(3s^{2}+t)+sx_{2}+x_{3}=0 \\
(x_{0}s-x_{1})(2x_{0}^{2}s^{2}-x_{0}x_{1}s+tx_{0}^{2}-x_{0}x_{2}-x_{1}^{2})=0
\end{array} \right.$ \\
$[0:0:s:t]$ & singular \rule{0pt}{13pt}\\
\bottomrule
\end{tabular}
\caption{points on $\esse={\Vi(x_{0}^{2}x_{3}^{}+x_{0}^{}x_{1}^{}x_{2}^{}+x_{1}^{3})}$.} \label{tabella0}
\end{table}
\begin{table}[hbt]
\centering
\begin{tabular}{ccc}
\toprule
coordinates of $\gr{a}$ & restrictions & $\TG_{\gr{a}}\esse \cap \esse$ (if smooth)\\
\toprule
$[1:t:-t^{2}s:s]$ & $s \neq 0 \neq t$ & line and irreducible conic\\
$[1:t:0:0]$ & $t \neq 0$ &
$\left\{\begin{array}{l}
x_{2}+x_{3}t^{2}=0 \\
x_{3}(x_{0}t\pm x_{1})=0
\end{array} \right.$ \rule{0pt}{20pt}\\
$[1:0:0:s]$ &  &
$\left\{\begin{array}{l}
x_{2}=0 \\
x_{1}^{2}x_{3}=0
\end{array} \right.$ \rule{0pt}{20pt}\\
$[0:1:t:0]$ &  &
$\left\{\begin{array}{l}
x_{3}=0 \\
x_{0}^{2}x_{2}=0
\end{array} \right.$ \rule{0pt}{20pt}\\
$[0:0:s:t]$ & & singular \rule{0pt}{13pt} \\
\bottomrule
\end{tabular}
\caption{points on $\esse=\Vi(x_{0}^{2}x_{2}^{}+x_{1}^{2}x_{3}^{})$.} \label{tabella}
\end{table}
If $\esse$ is projectively equivalent to ${\Vi(x_{0}^{2}x_{3}^{}+x_{0}^{}x_{1}^{}x_{2}^{}+x_{1}^{3})}$, then Table \ref{tabella0} shows that there are no T-points at all. If $\esse$ is projectively equivalent to ${\Vi(x_{0}^{2}x_{2}^{}+x_{1}^{2}x_{3}^{})}$, then $\mathcal{A}$ is contained in the line $[s:t:0:0]$ and $\mathcal{A}'$ is contained in the union of the lines $[s:0:0:t]$ and $[0:s:t:0]$, as shown in Table \ref{tabella}.\\

Now, let us assume $\gr{a}$ is smooth. Since it is not a T-point, $\TG_{\gr{a}}\esse$ cannot contain lines made up of T-points, so every such a line intersects $\TG_{\gr{a}}\esse$ in one and only one point. Since they are finite in number, the first statement of 1. is proved.\\
For the second statement, let $\gr{x} \neq \gr{a}$ be a point in $\pp^{3}_{\Kb}$ and let $Y=\Vi(P_{1,\gr{a}}) \cap \TG_{\gr{a}}\esse$. By \eqref{sviluppo}, the point $\gr{x} \in Y$ if and only if either $F(\gr{a}+t\gr{x})$ is the zero polynomial or the line through $\gr{a}$ and $\gr{x}$ intersects $\esse$ only in $\gr{a}$. This means that, if $\gr{x} \in Y$, also the whole line through it and $\gr{a}$ is contained in $Y$; the conclusion then holds if we prove that $Y$ is a curve, that is, $\Vi(P_{1,\gr{a}}) \nsupseteq \TG_{\gr{a}}\esse$.\\
In fact, $\gr{a}$ is not a T-point and so there exists a point $\gr{y}$ on $\esse \cap \TG_{\gr{a}}\esse$ such that the line $r$ through $\gr{y}$ and $\gr{a}$ does not lie on $\esse$. The line $r$ intersects $\esse$ in $\gr{a}$ with multiplicity two and in $\gr{y}$ with multiplicity one: this implies $\gr{y} \notin \Vi(P_{1,\gr{a}})$. Part 1. of the lemma is proved.\\

If $\gr{a}$ is singular, then by Lemma \ref{lemmacono} only finitely many planes through $\gr{a}$ contain a line on $\esse$ through $\gr{a}$. For any other choice $\pi$, the same argument of the smooth case holds, if we replace $\TG_{\gr{a}}\esse$ with $\pi$. This proves part 2. of the lemma.
\end{proof}
\end{lem}
\begin{proof}[Proof of Theorem \ref{teorema}
] We divide the proof into four steps.
\begin{enumerate}[leftmargin=0.3cm,itemindent=1.cm,labelwidth=\itemindent,labelsep=0.3cm,align=left]
\item[Step 1:] looking for the second point.\\
Either $\gr{a^{1}}$ is smooth or it is singular.\\

If $\gr{a^{1}}$ is smooth, then by hypotheses $\esse \cap \TG_{\gr{a^{1}}}\esse$ is a cubic curve, neither set-theoretically union of lines ($\gr{a^{1}}$ is not a T-point), nor the whole tangent plane ($\esse$ is irreducible).\\
Every line $\ell$ on $\TG_{\gr{a^{1}}}\esse$ through $\gr{a^{1}}$, but
those contained in $\TG_{\gr{a^{1}}}\esse \cap \Vi(P_{1,\gr{a^{1}}})$ as in Lemma \ref{lemmaTpoints},
has one and only one intersection with $\esse$ different from $\gr{a^{1}}$. Here we do not care about any line on $\TG_{\gr{a^{1}}}\esse \cap \esse$ through $\gr{a^{1}}$, since by Lemma \ref{lemmaTpoints} it would be contained in $\TG_{\gr{a^{1}}}\esse \cap \Vi(P_{1,\gr{a^{1}}})$ as well.\\
Fix a line $\ell$; the so-obtained $\gr{a^{2}}$ is smooth. Otherwise, $\ell$ would have multiplicity of intersection at least four with $\esse$, and therefore $\ell \subset \esse$, which is not.\\
Moreover, by Lemma \ref{lemmaTpoints}, $\gr{a^{2}}$ can be a T-point only for finitely many choices of $\ell$, and so these choices can be avoided.\\
By \eqref{sviluppo}, in coordinates we have, having chosen a representative $\so{a^{1}}$ for $\gr{a^{1}}$,
$$
\so{a^{2}}=F(\so{y})\cdot \so{a^{1}}-P_{1,\so{a^{1}}}(\so{y})\cdot \so{y},
$$
for any choice of $\so{y}=(y_{0},y_{1},y_{2},y_{3})$ representing the class $\gr{y} \in \TG_{\gr{a^{1}}}\esse$. Let us observe that $P_{1,\so{a^{1}}}(\so{y}) \neq 0$ and that $\so{a^{2}}$ has coordinates in $\mathbb{K}$.\\

If $\gr{a^{1}}$ is singular, the previous argument can be repeated by replacing the role of $\TG_{\gr{a^{1}}}\esse$ above with a plane $\pi$ satisfying Lemma \ref{lemmaTpoints}.\\

In both cases, we have constructed a smooth point $\gr{a^{2}}$ on $\esse$, which is not a T-point.
\item[Step 2:] looking for the third point.\\
The tangent plane process can be repeated as in step 1 --- smooth case --- starting from $\gr{a^{2}}$ to construct next point $\gr{a^{3}}$.
Summarizing, every line on $\TG_{\gr{a^{2}}} \esse$ through $\gr{a^{2}}$ with the exception of
\begin{itemize}
\item finitely many (by Lemma \ref{lemmaTpoints}) lines through T-points,
\item at most two lines in $\TG_{\gr{a^{2}}}\esse \cap \Vi(P_{1,\gr{a^{2}}})$ as in Lemma \ref{lemmaTpoints}
\end{itemize}
has exactly one intersection with $\esse$ different from $\gr{a^{2}}$, say $\gr{a^{3}}$. It is smooth and not a T-point.\\
To state that $\gr{a^{3}}$ is in general position with $\gr{a^{1}}$ and $\gr{a^{2}}$, we only need to verify that it does not lie on the line $\ell'$ through them. This is for free, since $\gr{a^{3}}$ belongs to $\TG_{\gr{a^{2}}}\esse$ but $\gr{a^{1}}$ does not, otherwise $\ell' \subseteq \esse$, which is not by construction.
\item[Step 3:] looking for the fourth point.\\
The tangent plane process can be repeated as in step 1 --- smooth case --- starting from $\gr{a^{3}}$ to construct next point $\gr{a^{4}}$. We need to choose it not on the plane $\pi_{123}$ containing $\gr{a^{1}}$, $\gr{a^{2}}$ and $\gr{a^{3}}$.\\
The planes $\TG_{\gr{a^{3}}}\esse$ and $\pi_{123}$ are distinct --- for example, the first one does not contain $\gr{a^{2}}$ by construction --- so their intersection is a line through $\gr{a^{3}}$, say $\ell''$.
\begin{cleim}
The system
\begin{equation}
\label{sistemas}
\left\{ \begin{array}{l}
\gr{y} \in \esse \\
\gr{y} \in \TG_{\gr{a^{3}}} \esse \\
\TG_{\gr{y}} \esse \ni \gr{a^{2}}
\end{array}
\right.
\end{equation}
which can be translated in homogeneous equations of degree $3,1,2$ respectively, has finitely many solutions $\gr{y} \in \pp^{3}_{\Kb}$.\\
Indeed, the system represents the intersection on the plane $\TG_{\gr{a^{3}}} \esse$ between the cubic curve $\mathcal{C}=\esse \cap \TG_{\gr{a^{3}}} \esse$ and the conic $\mathcal{Q}$ defined on $\TG_{\gr{a^{3}}} \esse$ by the condition ${\TG_{\gr{y}} \esse \ni \gr{a^{2}}}$. By construction, $\gr{a^{3}}$ is not a T-point and therefore $\mathcal{C}$ is either irreducible or union of a line and an irreducible conic containing $\gr{a^{3}}$; $\mathcal{Q}$ does not pass through $\gr{a^{3}}$ and so it cannot be contained in $\mathcal{C}$. This proves the claim.
\end{cleim}

The finitely many solutions of system \eqref{sistemas} correspond to finitely many lines on $\TG_{\gr{a^{3}}}\esse$ through $\gr{a^{3}}$. Since we want $\gr{a^{2}} \notin \TG_{\gr{a^{4}}} \esse$, we will avoid them.\\
Summarizing, every line on $\TG_{\gr{a^{3}}} \esse$ through $\gr{a^{3}}$ with the exception of
\begin{itemize}
\item finitely many lines through the solutions $\gr{y}$ of system \eqref{sistemas},
\item $\ell''$,
\item finitely many (by Lemma \ref{lemmaTpoints}) lines through T-points,
\item at most two lines in $\TG_{\gr{a^{3}}}\esse \cap \Vi(P_{1,\gr{a^{3}}})$ as in Lemma \ref{lemmaTpoints}
\end{itemize}
has exactly one intersection with $\esse$ different from $\gr{a^{3}}$, say $\gr{a^{4}}$. It is smooth and not a T-point, moreover $\gr{a^{2}} \notin \TG_{\gr{a^{4}}} \esse$.
\item[Step 4:] looking for the fifth point.\\
We can apply the usual tangent plane process to find $\gr{a^{5}}$ in general position with $\gr{a^{1}}, \gr{a^{2}}, \gr{a^{3}}$ and $\gr{a^{4}}$. Let us call $\pi_{ijk}$ the plane through different $\gr{a^{i}},\gr{a^{j}},\gr{a^{k}}$. The planes $\pi_{134}, \pi_{234}$ and $\pi_{124}$ intersect $\TG_{\gr{a^{4}}}\esse$ into three lines through $\gr{a^{4}}$: in fact they are four different planes, since $\gr{a^{2}},\gr{a^{3}} \notin \TG_{\gr{a^{4}}}\esse$.\\
The line $\pi_{123} \cap \TG_{\gr{a^{4}}}\esse$ cannot be contained in $\TG_{\gr{a^{4}}}\esse \cap \esse$, since $\gr{a^{4}} \notin \pi_{123}$ and by construction  $\gr{a^{4}}$ is not a T-point. This means that $\pi_{123} \cap \TG_{\gr{a^{4}}}\esse \cap \esse$ contains at most three points.\\
Summarizing, every line on $\TG_{\gr{a^{4}}} \esse$ through $\gr{a^{4}}$ with the exception of
\begin{itemize}
\item three lines lying on the planes $\pi_{134}, \pi_{234}$ and $\pi_{124}$,
\item at most three lines through the points in $\pi_{123} \cap \TG_{\gr{a^{4}}}\esse \cap \esse$,
\item at most two lines in $\TG_{\gr{a^{4}}}\esse \cap \Vi(P_{1,\gr{a^{4}}})$ as in Lemma \ref{lemmaTpoints}
\end{itemize}
has exactly one intersection with $\esse$ different from $\gr{a^{4}}$, say $\gr{a^{5}}$, in general position with $\gr{a^{1}}$, $\gr{a^{2}}$, $\gr{a^{3}}$, $\gr{a^{4}}$.
\end{enumerate}
\end{proof}
\begin{rmk}
Following the proof of Theorem \ref{teorema}, it is possible to implement an algorithm which requires a $\mathbb{K}$-point on $\esse$, not a T-point, and ensures five $\mathbb{K}$-points in general position on $\esse$. To test if a given point is a T-point or not, it is sufficient to check the reducibility of a polynomial of degree three in three variables, a task which can be easily performed with a software computation.
\end{rmk}
\begin{rmk}
If $\esse$ is a smooth cubic surface, then any T-point $P$ has $\TG_{P}\esse \cap \esse$ made up of three distinct lines. In such a situation, Theorem \ref{teorema} can be proved with the weaker hypothesis: the starting point $\gr{a^{1}}$ is not an Eckardt point.
\end{rmk}
\begin{rmk}
In the statement of Theorem \ref{teorema} we require that $\gr{a^{1}}$ is not a T-point. Indeed, if $\gr{a^{1}}$ is Eckardt, then the tangent plane process fails at the very first step. If $\gr{a^{1}}$ is a non-Eckardt T-point, then the tangent plane process could give rise to either singular or other T-points, which can make one loose control in subsequent steps.\\
In facts, this does happen in the following example: take $\esse=\Vi(x_{0}x_{1}x_{3}+x_{2}^{3}+x_{2}x_{3}^{2})$ and $\gr{a^{1}}=[0:0:0:1]$. The tangent plane process gives rise to points on the line $[s:t:0:0]$, which are either singular or Eckardt points. The process then stops at the second step.
\end{rmk}

As pointed out by the referee, codimension three AG subschemes have been considered also in \cite{MigliorePeterson}, where they are obtained as zero loci of sections of certain rank-three sheaves. In the case of five points in general position in $\pp^{3}_{\Kb}$, it turns out that all such sets are the zero loci of appropriate sections of the bundle $\Omega_{\pp^{3}}(3)$, which can be interpreted as four-tuple quadrics, that is, linear combinations (using linear forms as coefficients) of the syzygies of the map $(x_{0} \; x_{1} \; x_{2} \; x_{3})$. The membership of such a zero locus to a surface $\esse$ imposes conditions to the linear combination.
%
%
\section{Main results and further generalizations} \label{three}
In this last section, we firstly make use of Theorem \ref{teorema} and Algorithm \ref{algo} to prove Theorem \ref{mainthm}; if we drop the requirement of the starting point, then a weaker result holds (Proposition \ref{propouno}). After discussing the cases of reducible surfaces and cones, we state Theorem \ref{thmulti}. A concrete example is finally given.
\theoremstyle{plain}
\newtheorem*{thmole}{Theorem \ref{mainthm}}
\newtheorem*{thmultif}{Theorem \ref{thmulti}}
\newtheorem*{propoun}{Proposition \ref{propouno}}

\begin{proof}[Proof of Theorem \ref{mainthm}]
Given $\gr{a^{1}}$, one can apply Theorem \ref{teorema} and construct four other $\mathbb{K}$-points $\gr{a^{2}}$, $\gr{a^{3}}$, $\gr{a^{4}}$, $\gr{a^{5}}$ on $\esse$ such that they are all together in general position. With these initial data, Algorithm \ref{algo} ensures a Pfaffian $\mathbb{K}$-representation of $\esse$.
\end{proof}
%
\begin{rmk}
Let us work on $\Kb$ and let $\esse$ be general. In Remark \ref{tutteequiv} we saw that the Pfaffian representations produced by Algorithm \ref{algo} are all equivalent, once fixed the inputs $\gr{a^{1}},\gr{a^{2}},\gr{a^{3}},\gr{a^{4}},\gr{a^{5}}$. The constructive proof of Theorem \ref{mainthm} provides a new algorithm to construct many Pfaffian representations starting from just one point $\gr{a^{1}}$: we claim that neither this algorithm is surjective onto the possible Pfaffian representations of $\esse$, once fixed $\gr{a^{1}}$.\\
Indeed, by Remark \ref{neaspetto5}, the space of essentially different Pfaffian representations of $\esse$ is five-dimensional. Since we can suppose $\esse$ is smooth, $\gr{a^{1}}$ is not singular. The procedure described in the proof of Theorem \ref{teorema} consists  in taking a point on a plane cubic curve in each step, and so the space of sets of five points obtained starting from $\gr{a^{1}}$ is four-dimensional. The conclusion follows again by Remark \ref{tutteequiv}.
\end{rmk}
\begin{rmk}
The procedure lying beneath the proof of Theorem \ref{mainthm} involves only linear equations and can be implemented in a deterministic algorithm. 
\end{rmk}
\subsection{Weakening hypotheses}
\label{commenti}
\subsubsection{No starting points}
One of the hypotheses of Theorem \ref{mainthm} was a $\mathbb{K}$-point on $\esse$. If this is not given, then one can manage to find a $\mathbb{K}'$-point $\gr{a}$, being $\mathbb{K}'$ an algebraic extension of degree at most three, simply by solving a polynomial equation of degree three (given by intersecting $\esse$ with two arbitrary planes). For the general choice of these two planes, $\gr{a}$ is not a T-point and so Theorem \ref{mainthm} applies. This proves Proposition \ref{propouno}.
%
\subsubsection{Reducible surfaces}
Let $\esse$ be a reducible cubic surface. Then $\esse$ is either union of three planes with equation $\pi_{1},\pi_{2},\pi_{3}$ or union of a plane $\pi$ and a quadratic irreducible surface $\mathcal{S}$. In both cases, simple Pfaffian representations can be constructed, as we will show.\\
In the first case, a Pfaffian representation is given by
$
\left(
\begin{array}{c|c}
0 & M \\
\hline -M & 0
\end{array}
\right)$, where $$M = 
\left(
\begin{array}{ccc}
\pi_{1} & 0 & 0 \\
0 & \pi_{2} & 0 \\
0 & 0 & \pi_{3}
\end{array}
\right).
$$
In the second case, let us consider the matrix
\begin{equation*}
\mathbb{T}'=\aprimat{ccc}
0&-x_{3}&-x_2\\
x_{3}&0&-x_{1}\\
x_2&x_1&0
\chiudimat\mbox{.}
\end{equation*}
If $\mathcal{S} \ni [1:0:0:0]$, then we can find three linear forms $L_{1}, L_{2}, L_{3}$ such that an equation for $\mathcal{S}$ is $\sum_{i=1}^{3}(-1)^{i+1}L_{i}x_{i}$. A Pfaffian representation of $\mathcal{S}$ is then given by
$$P=\left(
\begin{array}{c|c}
\mathbb{T}' &
\begin{array}{c}
L_{1}\\L_{2}\\L_{3}
\end{array}
\\
\hline
\begin{array}{ccc}
-L_{1} & -L_{2} & -L_{3}
\end{array}
 & 0
\end{array}
\right)$$
by formula \eqref{formulapfaff}.

If $[1:0:0:0] \notin \mathcal{S}$, then it is sufficient to apply to $x_{1},x_{2},x_{3}$ in $\mathbb{T}'$ the projectivity which maps a given point $\gr{a}$ on $\mathcal{S}$ to $[1:0:0:0]$, as described in subsection \ref{analogo}. Again by formula \eqref{formulapfaff} one finds three linear forms and a Pfaffian representation $P$ of $\mathcal{S}$ as above.

A Pfaffian representation of $\esse$ is then given by
$$
\left(
\begin{array}{c|c|c}
0 & 0 & \pi \\
\hline
0 & P & 0 \\
\hline
-\pi & 0 & 0\\
\end{array}
\right).
$$
\begin{rmk}
\label{gradosei}
Let $F \in \mathbb{K}[x_{0},x_{1},x_{2},x_{3}]_{3}$ be an equation for the reducible surface $\esse$. The Pfaffian representations just constructed are not $\mathbb{K}$-representations, in general. This is due to the fact that the splitting field of a polynomial of degree three is generally an algebraic extension of $\mathbb{K}$ of degree six.\\
However, for such reducible surfaces we can state: it is possible to construct explicitly a Pfaffian $\mathbb{K}'$-representation, being $\mathbb{K'}$ an algebraic extension of $\mathbb{K}$ of degree at most six.
\end{rmk}
\subsubsection{Cones}
Let $\esse$ be an irreducible cone. If we suppose non-restrictively that its vertex is $[1:0:0:0]$, then $\esse$ is defined by an equation $F \in \mathbb{K}[x_{1},x_{2},x_{3}]$. Let us call $\mathcal{C}$ the plane cubic curve defined by $F$ in $\pp^{3}_{\bar{\mathbb{K}}} \cap \Vi(x_{0})$.\\
As previously done, we can find a $\mathbb{K}'$-point $\gr{a}$ on $\mathcal{C}$, being $\mathbb{K}'$ an algebraic extension of $\mathbb{K}$, simply by solving a polynomial equation of degree three.\\
The construction of $\mathbb{K}'$-points on a plane cubic curve is a subject widely studied in ,literature (see for example \cite{SilvermanTate}). Starting from a set $X$ of $\mathbb{K}'$-points, it consists in considering tangent lines to the curve in each point of $X$, and secant lines through each pair of points of $X$; the third intersection of such lines with $\mathcal{C}$ is then set as a new element in $X$.\\
This process fails, for particular choices of $X=\{\gr{a}\}$: for example, if $\gr{a}$ is an inflection point of the curve. For a general choice of $\gr{a}$, this process produces a lot of $\mathbb{K}'$-points on $\mathcal{C}$, and we can manage to find five points among them such that no three are collinear. Then the following proposition applies.
\begin{propo}
\label{suiconi}
Let $\esse$ be a cone over a plane cubic curve $\mathcal{C}$, with equation $F \in \mathbb{K}'[x_{0},x_{1},x_{2},x_{3}]_{3}$. If there exist five $\mathbb{K}'$-points on $\mathcal{C}$ such that no three of them are on a line, then there exist five $\mathbb{K}'$-points in general position on $\esse$.
\begin{proof}
We can suppose the vertex is $[1:0:0:0]$, so that the equation of the plane curve (and the cone) is $C=C(x_{1},x_{2},x_{3})$. Let $\so{a^{i}}=(a_{0}^{i},a_{1}^{i},a_{2}^{i},a_{3}^{i})$ represent the five points. The vanishing of each of the $4 \times 4$ minors of the matrix
\begin{equation}
\label{matricelift}
\aprimat{cccc}
y_{1} & a^{1}_{1} & a^{1}_{2} & a^{1}_{3} \\
y_{2} & a^{2}_{1} & a^{2}_{2} & a^{2}_{3} \\
y_{3} & a^{3}_{1} & a^{3}_{2} & a^{3}_{3} \\
y_{4} & a^{4}_{1} & a^{4}_{2} & a^{4}_{3} \\
y_{5} & a^{5}_{1} & a^{5}_{2} & a^{5}_{3} \\
\chiudimat
\end{equation}
imposes a non-trivial close condition to $\so{y} \in \mathbb{A}^{5}_{\mathbb{K}'}$, since the $3 \times 3$ minors of the matrix obtained  by deleting the first column in \eqref{matricelift} are non-vanishing by hypotheses. So there exists $\so{y}$ satisfying none of these conditions and we get five points in general positions on $\esse$.
\end{proof}
\end{propo}
Let us remark that also in the case of cones it is possible to implement an algorithm which requires an equation $F \in \mathbb{K}[x_{0},x_{1},x_{2},x_{3}]$ for the surface $\esse$ and ensures a Pfaffian $\mathbb{K'}$-representation of $\esse$, being $[\mathbb{K}':\mathbb{K}]\leq 3$.\\
Summarizing, we can prove Theorem \ref{thmulti}, which is a generalization of Proposition \ref{ognisuperficie}.
\begin{proof}[Proof of Theorem \ref{thmulti}]
It follows from Proposition \ref{propouno}, Remark \ref{gradosei} and from the discussion about cones made above.
\end{proof}
%
%
%
%
\subsection{An example}
Let $F=x_{0}x_{1}^{2}+x_{1}x_{3}^{2}+x_{2}^{3}$ be the equation of $\esse$, the unique cubic surface which does not admit a linear determinantal representation by \cite{BrunduLogar}, up to projectivity. Let us consider the point $\gr{a^{1}}=[1:0:0:0]$, which is singular and therefore not a T-point. Then Theorem \ref{mainthm} applies, and we can construct explicitly a Pfaffian $\mathbb{Q}$-representation of $\esse$.\\
According to the proof of Theorem \ref{teorema}, we choose the plane $x_{3}=0$, which does not cut $\esse$ in three lines. Considering the point $[1:1:0:0]$, the line through it and $\gr{a^{1}}$ intersects $\esse$ in $\gr{a^{2}}=[0:1:0:0]$.\\
We have
$$\TG_{\gr{a^{2}}}\esse \cap \esse : \left\{ \begin{array}{l} x_{0}=0 \\ x_{1}x_{3}^{2}+x_{2}^{3}=0 \end{array} \right.$$
and so we choose a point on $x_{0}=0$, say $[0:0:1:1]$. The line through it and $\gr{a^{2}}$ intersects $\esse$ in $\gr{a^{3}}=[0:-1:1:1]$.\\
We have
$$
\TG_{\gr{a^{3}}}\esse \cap \esse : \left\{ \begin{array}{l} x_{0}+x_{1}+3x_{2}-2x_{3}=0 \\ -x_{1}^{3}-3x_{1}^{2}x_{2}+2x_{1}^{2}x_{3}+x_{1}x_{3}^{2}+x_{2}^{3}=0 \end{array} \right.
$$
and so we choose a point satisfying the first equation, say $[5:0:-1:1]$. The  line through it and $\gr{a^{3}}$ intersects $\esse$ in $\gr{a^{4}}=[-10:1:1:-3]$.\\
We have
$$
\TG_{\gr{a^{4}}}\esse \cap \esse : \left\{ \begin{array}{l} x_{0}-11x_{1}+3x_{2}-6x_{3}=0 \\ 11x_{1}^{3}-3x_{1}^{2}x_{2}+6x_{1}^{2}x_{3}+x_{1}x_{3}^{2}+x_{2}^{3}=0 \end{array} \right.
$$
and so we choose a point satisfying the first equation, say $[40:2:-2:2]$. The  line through it and $\gr{a^{4}}$ intersects $\esse$ in $\gr{a^{5}}=[95:1:-6:11]$.\\
A Pfaffian $\mathbb{Q}$-representation can be obtained via Algorithm \ref{algo}. For example, simplifying denominators, we have $P=(p_{ij})$ with the following entries:
$$
\begin{array}{lcl}
p_{12} = 0, & \quad &
p_{13} = x_{2}-x_{3},\\
p_{14} = 0, & \quad &
p_{15} = 3x_{2}+x_{3},\\
p_{16} = 1470x_{1}+686x_{2}+588x_{3}, & \quad &
p_{23} = -x_{2}+x_{3},\\
p_{24} = 34x_{0}-510x_{1}-170x_{2}-340x_{3}, & \quad &
p_{25} = 2x_{1}+x_{2}+x_{3},\\
p_{26} = 1372x_{1}+588x_{3}, & \quad &
p_{34} = 8670x_{1}+6120x_{2}+2550x_{3},\\
p_{35} = -34x_{1}-17x_{2}-17x_{3}, & \quad &
p_{36} = -23324x_{1}-10829x_{3},\\
p_{45} = 0, & \quad &
p_{46} = 774690x_{1}-624750x_{2},\\
p_{56} = -21658x_{1}+11662x_{2}+833x_{3}.
\end{array}
$$

\begin{ack}
The author thanks his supervisor Emilia Mezzetti for her constant support and wise pieces of advice; he also thanks Alessandro Logar for many useful conversations and Daniele Faenzi for interesting discussions. In addition, he would like to thank the referee for helpful comments.
\end{ack}


\begin{thebibliography}{10}

\bibitem{Abhyankar}
S.~Abhyankar.
\newblock Cubic surfaces with a double line.
\newblock {\em Mem. Coll. Sci. Univ. Kyoto. Ser. A. Math.}, 32:455--511, 1960.

\bibitem{AdlerRamanan}
A.~Adler and S.~Ramanan.
\newblock {\em Moduli of abelian varieties}, volume 1644 of {\em Lecture Notes
  in Mathematics}.
\newblock Springer-Verlag, Berlin, 1996.

\bibitem{Bass}
H.~Bass.
\newblock On the ubiquity of {G}orenstein rings.
\newblock {\em Math. Z.}, 82:8--28, 1963.

\bibitem{Beauville}
A.~Beauville.
\newblock Determinantal hypersurfaces.
\newblock {\em Michigan Math. J.}, 48:39--64, 2000.
\newblock Dedicated to William Fulton on the occasion of his 60th birthday.

\bibitem{BrambillaFaenzi}
M.~C. Brambilla and D.~Faenzi.
\newblock Moduli spaces of rank-2 {ACM} bundles on prime {F}ano threefolds.
\newblock {\em Michigan Math. J.}, 60(1):113--148, 2011.

\bibitem{BrunduLogar}
M.~Brundu and A.~Logar.
\newblock Parametrization of the orbits of cubic surfaces.
\newblock {\em Transform. Groups}, 3(3):209--239, 1998.

\bibitem{BuchsbaumEisenbud}
D.~A. Buchsbaum and D.~Eisenbud.
\newblock Gorenstein ideals of height {$3$}.
\newblock In {\em Seminar {D}. {E}isenbud/{B}. {S}ingh/{W}. {V}ogel, {V}ol. 2},
  volume~48 of {\em Teubner-Texte zur Math.}, pages 30--48. Teubner, Leipzig,
  1982.

\bibitem{Buckley}
A.~Buckley.
\newblock Elementary transformations of {P}faffian representations of plane
  curves.
\newblock {\em Linear Algebra Appl.}, 433(4):758--780, 2010.

\bibitem{BuckleyKosirPla}
A.~Buckley and T.~Ko{\v{s}}ir.
\newblock Plane curves as {P}faffians.
\newblock {\em Ann. Sc. Norm. Super. Pisa Cl. Sci. (5)}, 10(2):363--388, 2011.

\bibitem{ChiantiniFaenziOnG}
L.~Chiantini and D.~Faenzi.
\newblock On general surfaces defined by an almost linear {P}faffian.
\newblock {\em Geom. Dedicata}, 142:91--107, 2009.

\bibitem{Conforto}
F.~Conforto.
\newblock {\em Le superficie razionali}.
\newblock Zanichelli Editore, Bologna, 1939.
%

\bibitem{CoskunKulkarniMustopa}
E.~Coskun, R.~S. Kulkarni, and Y.~Mustopa.
\newblock Pfaffian quartic surfaces and representations of clifford algebras.
\newblock E-print arXiv:1107.1522.
\newblock 7/2011.

\bibitem{DavisGeramitaOrecchia}
E.~D. Davis, A.~V. Geramita, and F.~Orecchia.
\newblock Gorenstein algebras and the {C}ayley-{B}acharach theorem.
\newblock {\em Proc. Amer. Math. Soc.}, 93(4):593--597, 1985.

\bibitem{Dolgachev}
I.~V. Dolgachev.
\newblock {\em Classical algebraic geometry: a modern view}.
\newblock To be published by Cambridge University Press, 2012.

\bibitem{Eisenbud}
D.~Eisenbud.
\newblock {\em Commutative algebra}, volume 150 of {\em Graduate Texts in
  Mathematics}.
\newblock Springer-Verlag, New York, 1995.
\newblock With a view toward algebraic geometry.

\bibitem{FaenziARe}
D.~Faenzi.
\newblock A remark on {P}faffian surfaces and a{CM} bundles.
\newblock In {\em Vector bundles and low codimensional subvarieties: state of
  the art and recent developments}, volume~21 of {\em Quad. Mat.}, pages
  209--217. Dept. Math., Seconda Univ. Napoli, Caserta, 2007.

\bibitem{FaniaMezzetti}
M.~L. Fania and E.~Mezzetti.
\newblock On the {H}ilbert scheme of {P}alatini threefolds.
\newblock {\em Adv. Geom.}, 2(4):371--389, 2002.

\bibitem{Han}
F.~Han.
\newblock Pfaffian bundles on cubic surfaces and configurations of planes.
\newblock E-print arXiv:1210.1763.
\newblock 10/2012.

\bibitem{IlievMarkushevich}
A.~Iliev and D.~Markushevich.
\newblock Quartic 3-fold: {P}faffians, vector bundles, and half-canonical
  curves.
\newblock {\em Michigan Math. J.}, 47(2):385--394, 2000.

\bibitem{ManinCub}
Y.~I. Manin.
\newblock {\em Cubic forms}, volume~4 of {\em North-Holland Mathematical
  Library}.
\newblock North-Holland Publishing Co., Amsterdam, second edition, 1986.
\newblock Algebra, geometry, arithmetic, Translated from the Russian by M.
  Hazewinkel.
  
\bibitem{MigliorePeterson}
J.~C. Migliore and C.~Peterson.
\newblock A construction of codimension three arithmetically {G}orenstein
  subschemes of projective space.
\newblock {\em Trans. Amer. Math. Soc.}, 349(9):3803--3821, 1997.

\bibitem{Mordell}
L.~J. Mordell.
\newblock {\em Diophantine equations}.
\newblock Pure and Applied Mathematics, Vol. 30. Academic Press, London, 1969.

\bibitem{Segre}
B.~Segre.
\newblock On the rational solutions of homogeneous cubic equations in four
  variables.
\newblock {\em Math. Notae}, 11:1--68, 1951.

\bibitem{SilvermanTate}
J.~H. Silverman and J.~Tate.
\newblock {\em Rational points on elliptic curves}.
\newblock Undergraduate Texts in Mathematics. Springer-Verlag, New York, 1992.

\bibitem[CoCoA]{CoCoA}
\textsc{CoCoATeam}.
\newblock {\em CoCoA: a system for doing Computations in Commutative Algebra}.
\newblock Available at \textsf{http://cocoa.dima.unige.it}.

\end{thebibliography}

\end{document}